\documentclass[11pt]{article}
\usepackage[a4paper,margin=28mm]{geometry}
\usepackage[T1]{fontenc}
\usepackage{lmodern}
\usepackage{amsmath,amssymb,amsthm,mathtools}
\usepackage{enumitem}
\usepackage{microtype}
\usepackage{tikz}
\usepackage[numbers,sort&compress]{natbib}
\usepackage[colorlinks=true,linkcolor=blue,citecolor=blue,urlcolor=blue]{hyperref}

\newtheorem{theorem}{Theorem}[section]
\newtheorem{proposition}[theorem]{Proposition}
\newtheorem{lemma}[theorem]{Lemma}
\theoremstyle{definition}
\newtheorem{definition}[theorem]{Definition}
\newtheorem{example}[theorem]{Example}
\newtheorem{question}[theorem]{Question}
\theoremstyle{remark}
\newtheorem{remark}[theorem]{Remark}

\newcommand{\one}{\mathbf 1}
\newcommand{\R}{\mathbb R}
\newcommand{\cA}{\mathcal A}
\newcommand{\hmu}{\widehat\mu}
\newcommand{\simm}{\sim_{\mathrm m}}
\newcommand{\simc}{\sim_{\mathrm c}}
\newcommand{\Cq}{\mathsf C}
\DeclareMathOperator{\Aut}{Aut}
\DeclareMathOperator{\diag}{diag}

\title{Adjacency-degree algebras and spectral determination of graphs}
\author{Zhipeng Lu\thanks{Shenzhen MSU--BIT University and Guangdong Laboratory
of Machine Perception and Intelligent Computing, Shenzhen, China.
Email: \texttt{zhipeng.lu@hotmail.com}.}
\and Pengxiang Li\thanks{Beijing Institute of Technology, Beijing, China.
Email: \texttt{pengxiangli@bit.edu.cn}.}}
\date{}

\begin{document}
\maketitle

\begin{abstract}
Let $G$ be a graph with adjacency matrix $A$ and degree matrix $D$.
Applying $A$ and $D$ repeatedly to the all-ones vector produces a space
$M_G$ of functions on the vertices, and these functions are constant on
the orbits of the automorphism group. We call $G$ AD-observable if $M_G$
contains the indicator vector of every orbit, and AD-separable if the
functions in $M_G$ separate the orbits. Forests are AD-observable, by
McKay's leaf-pruning argument. Our main result is that for every cograph,
$M_G$ is spanned by the indicator vectors of the cells of the stable
partition of color refinement. Hence a cograph is AD-observable, or
equivalently AD-separable, exactly when color refinement finds its orbits,
and we describe these cographs by their cotrees. We also show that two
cographs with the same number of vertices and the same homomorphism counts
from all caterpillars have the same counts from all trees. The main result
does not extend to distance-hereditary graphs: we give two such graphs on
nine vertices that are AD-separable but not AD-observable.
\end{abstract}



\section{Introduction}

Counting walks is a classical way of telling the vertices of a
graph apart. If $A$ is the adjacency matrix of a graph $G$ and $\one$ is
the all-ones vector, then $(A^k\one)_v$ is the number of walks of length
$k$ starting at the vertex $v$. Vertices in the same orbit of the
automorphism group have the same walk counts. The converse fails badly:
in a regular graph all vertices have the same walk counts, whether or not
the graph is vertex-transitive.

Walks that also record degrees do better. Let $D$ be the diagonal matrix
of vertex degrees, and let
\[
  M_G=\R\langle A,D\rangle\one
\]
be the smallest subspace of $\R^{V(G)}$ that contains $\one$ and is
invariant under $A$ and $D$. The entry of the vector
$D^{a_0}AD^{a_1}A\cdots AD^{a_\ell}\one$ at $v$ is the number of
homomorphisms to $G$ that send $x_0$ to $v$, from the caterpillar with
spine $x_0x_1\cdots x_\ell$ and $a_i$ pendant leaves at $x_i$. Every
vector in $M_G$ is constant on orbits, so $M_G$ lies in the space $U_G$
spanned by the indicator vectors of the orbits. We call $G$
\emph{AD-observable} if $M_G=U_G$, and \emph{AD-separable} if any two
vertices in different orbits take different values on some vector of
$M_G$. Observability implies separability. The converse fails, because a
space of functions can separate points without containing their indicator
vectors; the span of $(1,1,1)$ and $(0,1,2)$ is an example.

This paper asks which graphs have these properties. The question has two
parts. Color refinement colors each vertex by its degree, then by the
multiset of colors of its neighbours, and so on until the partition
stops changing. The resulting stable partition is the coarsest equitable
partition of $G$. Let $W_G$ be the span of the indicator vectors of its
cells. Then
\[
  M_G\subseteq W_G\subseteq U_G .
\]
Graphs with $W_G=U_G$, on which color refinement computes the orbits, are
called refinable~\cite{AKRV2017}. An AD-separable graph is refinable,
because the vectors of $M_G$ are constant on stable cells. The equality
$M_G=W_G$ is a separate matter. By results of
Dvo\v{r}\'ak~\cite{Dvorak2010} and of Dell, Grohe and
Rattan~\cite{DGR2018}, $W_G$ is spanned by the vectors counting
homomorphisms from rooted trees. (These vectors separate the stable cells,
and the entrywise product of two of them counts homomorphisms from the two
trees glued at their roots.) So $M_G=W_G$ asks whether every such vector
is a linear combination of the caterpillar vectors above, whose root is an
end of the spine. In general it is not; see Section~\ref{sec:beyond}.
When $M_G=W_G$ holds, the vectors of $M_G$
separate exactly the stable cells, and then AD-observability,
AD-separability and refinability are the same property.

Forests are AD-observable. This follows from McKay's matrix form of leaf
pruning~\cite{McKay1977}, and we give the argument in
Section~\ref{sec:forests}. A tree is built from a single vertex by
adding pendant vertices one at a time. The other basic way of adding a
vertex is to add a twin, that is, a new vertex adjacent to exactly the
neighbours of an old vertex and possibly to the old vertex itself. The
graphs obtained from a single vertex by adding twins are exactly the
cographs~\cite{CorneilLerchsStewart1981}, and the connected graphs obtained
by adding pendant vertices and twins are exactly the connected
distance-hereditary graphs~\cite{BandeltMulder1986}. Our main result
concerns cographs.

\begin{theorem}\label{thm:intro-main}
Every cograph $G$ satisfies $M_G=W_G$.
\end{theorem}

Consequently a cograph is AD-observable if and only if it is AD-separable,
and if and only if color refinement finds its orbits. A cograph has a
canonical tree representation, its cotree, and refinability can be read
off from it.

For a node $x$ of the cotree let $G_x$ be the subgraph induced by the
leaves below $x$.

\begin{theorem}\label{thm:intro-observable}
A cograph is AD-observable, or equivalently AD-separable, if and only if
every internal node $x$ of its reduced cotree has the following property:
if $y$ and $z$ are children of $x$ and some stable color of $G_x$ occurs
both in $G_y$ and in $G_z$, then $G_y\cong G_z$.
\end{theorem}

For instance, every threshold graph and every complete multipartite graph
is AD-observable, while the $4$-regular cograph $\overline{K_3\sqcup C_4}$
on seven vertices is not.

The proof of Theorem~\ref{thm:intro-main} works with pairs of graphs. For a
graph $G$ on $n$ vertices and a word $w$ in two letters, put
$\hmu_G(w)=\one^Tw(A,D)\one/n$. We show that two connected cographs with
the same normalized moments $\hmu$ have the same color refinement quotient,
with the same relative cell sizes (Proposition~\ref{prop:key}). Since
$n\hmu_G(w)$ is the number of homomorphisms from a caterpillar to $G$, this
has the following consequence.

\begin{theorem}\label{thm:intro-hom}
Let $G$ and $H$ be cographs with the same number of vertices. If
$\hom(C,G)=\hom(C,H)$ for every caterpillar $C$, then
$\hom(T,G)=\hom(T,H)$ for every tree $T$.
\end{theorem}

Neither Theorem~\ref{thm:intro-main} nor Theorem~\ref{thm:intro-hom}
holds for all graphs. In Section~\ref{sec:beyond} we give two graphs on
ten vertices, related by a switching, with the same caterpillar counts,
although color refinement tells them apart; one of them is refinable but
not AD-separable. The equality $M_G=W_G$ fails for two distance-hereditary
graphs on nine vertices, which are AD-separable but not AD-observable, and
a computer search shows that among distance-hereditary graphs with at most
nine vertices these are the only failures. Thus $M_G=W_G$ holds for all
forests and all cographs, but not for all distance-hereditary graphs.

\medskip
The natural approach to Theorem~\ref{thm:intro-main} is induction over the
cotree, using the fact that $M_G$ does not change under complementation.
The difficulty is that two connected cographs with the same normalized
moments can have different orders. For example, with $F=K_2\sqcup K_1$
the graphs $F\vee F\vee F$ and $(F\sqcup F)\vee(F\sqcup F)$ have nine and
twelve vertices and the same normalized moments. Complementing such a
pair does not preserve moment equivalence, because the formula for the
complement involves the number of vertices. We therefore argue by
induction on $\dim M_G$, and follow the cotree downwards in each graph
until components of two different moment types appear; these components
have smaller walk spaces. The moments at that stage are obtained from
those of the original graph by one of two transformations, determined by
the smallest nonzero and by the largest eigenvalue of the Laplacian on
$M_G$, and the induction goes through unless the two graphs end with
different transformations. Ruling this out is the main step. It reduces
to a statement about all graphs (Theorem~\ref{thm:wrong-size}): if $F$
and $H$ both have components of at least two moment types, then the
normalized moments of $H$ cannot be obtained from those of $F$ by the
complement formula, whatever positive number is used in it in place of
$|V(F)|$. The proof uses only the positivity of the signless Laplacian and
the Perron--Frobenius theorem.

\subsection*{Related work}
The space $\R[A]\one$ and its dimension, the number of main eigenvalues,
are classical; see Cvetkovi\'c~\cite{Cvetkovic1978}, Powers and
Sulaiman~\cite{PowersSulaiman1982}, and the survey of
Rowlinson~\cite{Rowlinson2007}. A graph with $\R[A]\one=\R^{V(G)}$ is
called controllable. Godsil conjectured, and O'Rourke and Touri proved,
that almost all graphs are controllable~\cite{Godsil2012,ORourkeTouri2016}.
Controllable graphs are AD-observable, so almost all graphs are. For
trees one needs $D$ as well: McKay~\cite{McKay1977} proved that the
spectra of the polynomials in $A$ and $D$ determine a tree, and Li, Fan,
Ito, Karimi and Xu~\cite{LiFanIto2021} treated the corresponding question
for the Terwilliger algebra of a rooted tree. Our forest theorem is a
reorganisation of McKay's argument.

The connection between homomorphism counts and color refinement goes
back to Dvo\v{r}\'ak~\cite{Dvorak2010} and to Dell, Grohe and
Rattan~\cite{DGR2018}. Grohe, Rattan and Seppelt~\cite{GRS2025} describe
many equivalences of this kind by systems of linear equations, and
caterpillar counts, together with their description by walks with degree
colors, appear in recent work of \v{C}ern\'y~\cite[Theorem~4.1]{Cerny2025}.
On general graphs caterpillar counts are strictly weaker than tree
counts; see~\cite[Lemma~B.13]{Cerny2025}, or apply the theorem of Neuen
and Seppelt~\cite[Theorem~1.5]{NeuenSeppelt2026} on homomorphism
distinguishing closed classes to the forests whose components are
caterpillars. Section~\ref{sec:beyond} contains an explicit pair.
Theorem~\ref{thm:intro-hom} says that on cographs the two kinds of counts
carry the same information.

Refinable, compact and amenable graphs are compared in~\cite{AKRV2017}.
Tinhofer studied the graphs for which every doubly stochastic matrix
commuting with $A$ is a convex combination of automorphisms, now called
compact graphs~\cite{Tinhofer1986}, and proved that strong tree-cographs
are compact~\cite{Tinhofer1989}. The cotree criterion for refinability
behind Theorem~\ref{thm:intro-observable} (Proposition~\ref{prop:refinable})
is elementary, and we include its short proof.

Degree-similar graphs, for which one matrix conjugates both $A$ and $D$ of
one graph to those of another, have been studied
recently~\cite{GodsilSun2025,FanXingZhangWang2026,ZhangFan2026}. There
the whole pair $(A,D)$ is compared, whereas we only use the part
generated from $\one$.

\subsection*{Organisation}
Section~\ref{sec:prelim} collects definitions and basic facts.
Section~\ref{sec:forests} treats forests and Section~\ref{sec:components}
describes $M_G$ for disconnected graphs. Section~\ref{sec:wrong-size}
proves the statement about complements mentioned above.
Section~\ref{sec:cographs} proves Theorems~\ref{thm:intro-main}
and~\ref{thm:intro-hom}, and Section~\ref{sec:observable} proves
Theorem~\ref{thm:intro-observable}. Section~\ref{sec:beyond} discusses
graphs that are not cographs and lists some questions.

\section{Preliminaries}\label{sec:prelim}

Graphs are finite and simple and have at least one vertex; they need not
be connected. For a graph $G$ with vertex set $V$ and $n=|V|$ we write
$A=A_G$, $D=D_G$ and $\one=\one_V$, and we let $\cA(G)$ be the algebra of
real $V\times V$ matrices generated by $I$, $A$ and $D$. It is closed
under transposition. We put
\[
  M_G=\cA(G)\one .
\]
For a word $w$ in two letters we write $w(A,D)$ for the matrix obtained
by substituting $A$ and $D$, and we call the numbers
$\mu_G(w)=\one^Tw(A,D)\one$ the \emph{principal moments} of $G$. For a set
$X\subseteq V$ we write $\one_X$ for its indicator vector and
$E_X=\diag(\one_X)$.

Let $U_G$ be the space of vectors that are constant on the orbits of
$\Aut(G)$, and let $W_G$ be the span of the indicator vectors of the
cells of the stable partition of color refinement.

\begin{lemma}\label{lem:chain}
$M_G\subseteq W_G\subseteq U_G$.
\end{lemma}

\begin{proof}
The stable partition is equitable and its cells have constant degree, so
$W_G$ is invariant under $A$ and $D$; it contains $\one$, hence $M_G$. The
orbit partition is equitable, and every equitable partition refines the
stable partition, which is the coarsest one~\cite{Tinhofer1986,RamanaScheinermanUllman1994}.
So every stable cell is a union of orbits.
\end{proof}

Two vertices $u,v$ are \emph{AD-equivalent} if $f(u)=f(v)$ for every
$f\in M_G$. By Lemma~\ref{lem:chain}, every stable cell lies in one
AD-equivalence class.

\begin{definition}
A graph $G$ is \emph{AD-observable} if $M_G=U_G$, and \emph{AD-separable}
if the AD-equivalence classes are the orbits of $\Aut(G)$.
\end{definition}

An AD-observable graph is AD-separable, and an AD-separable graph is
refinable, that is, satisfies $W_G=U_G$: the partition into
AD-equivalence classes is coarser than the stable partition, which is
coarser than the orbit partition. Moreover $G$ is AD-observable if and only
if $M_G=W_G$ and $G$ is refinable. If $M_G=W_G$, the AD-equivalence classes
are the stable cells, and the three properties coincide.

The equality $M_G=W_G$ is a closure property: it holds if and only if
$M_G$ is closed under entrywise multiplication. Indeed, $W_G$ is so
closed. Conversely, a subspace containing $\one$ and closed under entrywise
products is spanned by the indicator vectors of the classes of vertices on
which all its vectors agree (apply polynomial interpolation entrywise to
one vector that separates the classes). If the subspace is $M_G$, these
classes form an equitable partition, because $M_G$ is $A$-invariant. This
partition refines the stable partition, so every stable cell is a union of
classes and $W_G\subseteq M_G$.

\begin{example}\label{ex:elementary}
If $G$ is $d$-regular, then $A\one=d\one$ and $D=dI$, so $M_G=\R\one$. A
regular graph is therefore AD-observable if and only if it is
AD-separable, and if and only if it is vertex-transitive. At the other
extreme, if $\R[A]\one=\R^V$ then
$M_G=\R^V$ and $G$ is AD-observable. By the theorem of O'Rourke and
Touri~\cite{ORourkeTouri2016}, the random graph $G(n,1/2)$ has this
property with probability tending to one.
\end{example}

\subsection{Normalized moments and pointed data}
For a graph $G$ on $n$ vertices put $s_G=\one/\sqrt n$ and
\[
  \hmu_G(w)=s_G^Tw(A,D)s_G=\frac{\mu_G(w)}{n}.
\]
We extend $\hmu_G$ linearly to noncommutative polynomials. Two graphs are
\emph{moment equivalent}, written $G\simm H$, if $\hmu_G(w)=\hmu_H(w)$
for every word $w$. For example, all regular graphs of the same degree are
moment equivalent. The following is standard; compare
\cite[Section~3]{GRS2025}.

\begin{lemma}\label{lem:pointed}
$G\simm H$ if and only if there is a linear isometry $\Phi$ from $M_G$
onto $M_H$ with $\Phi s_G=s_H$, $\Phi A_G=A_H\Phi$ and $\Phi D_G=D_H\Phi$
on $M_G$. Such a $\Phi$ is unique.
\end{lemma}

\begin{proof}
For polynomials $p,q$ we have
$\langle p(A,D)s_G,q(A,D)s_G\rangle=\hmu_G(p^*q)$, where $p^*$ reverses
every word. If the moments agree, $p(A_G,D_G)s_G\mapsto p(A_H,D_H)s_H$ is
therefore well defined and isometric, and it has the required properties.
The converse is clear, and uniqueness holds because the vectors
$p(A_G,D_G)s_G$ span $M_G$.
\end{proof}

We write $\Theta_G$ for the triple $(A|_{M_G},D|_{M_G},s_G)$, considered up
to isometries as in Lemma~\ref{lem:pointed}, and call it the
\emph{pointed data} of $G$. Thus $G\simm H$ means $\Theta_G\cong\Theta_H$.
We put $q(G)=\dim M_G$.

\subsection{Normalized color equivalence}
Let $C_1,\dots,C_k$ be the stable cells of $G$, let $B_{ij}$ be the number
of neighbours in $C_j$ of a vertex of $C_i$, and let $\pi_i=|C_i|/n$. We
write $\Cq(G)$ for the pair $(B,\pi)$, up to simultaneous permutation of
the indices, and call it the \emph{normalized quotient} of $G$. Two graphs,
possibly of different orders, are \emph{color equivalent}, written
$G\simc H$, if $\Cq(G)=\Cq(H)$. For
graphs of the same order this is the usual equivalence under color
refinement, or fractional isomorphism~\cite{Tinhofer1986,RamanaScheinermanUllman1994}.
The normalized quotient determines every normalized moment, since
$w(A,D)\one$ is constant on stable cells and is computed by the same word
in $B$ and $\diag(B\one)$. Hence $G\simc H$ implies $G\simm H$.

\begin{lemma}\label{lem:cr-facts}
\begin{enumerate}[label=(\alph*)]
\item If $G$ is the disjoint union or the join of graphs $G_1,\dots,G_k$,
then the restriction of the stable partition of $G$ to $V(G_i)$ is the
stable partition of $G_i$.
\item Let $H$ and $K$ be connected. If some stable color of $H\sqcup K$
occurs in both $H$ and $K$, then $H\simc K$, and every stable color of
$H\sqcup K$ occurs in both, with the same proportions.
\item If $H_1,\dots,H_k$ are connected and pairwise color equivalent, then
$\Cq(H_1\sqcup\cdots\sqcup H_k)=\Cq(H_1)$.
\end{enumerate}
\end{lemma}

\begin{proof}
(a) The union of the stable partitions of the $G_i$ is an equitable
partition of $G$; for a join, every vertex of $G_i$ has all vertices
outside $V(G_i)$ as neighbours. So the stable partition of $G$ is coarser,
and its restriction to $V(G_i)$ is coarser than the stable partition of
$G_i$. That restriction is also equitable in $G_i$, since vertices of $G_i$
have the same neighbours outside $V(G_i)$, so it is finer.

(b) A stable color determines the colors of the neighbours with
multiplicities. Since $H$ and $K$ are connected, the set of colors occurring
in both is closed under passing to neighbours and is therefore everything.
If colors $a$ and $b$ are adjacent, counting edges between them gives
$|a\cap V(H)|\,B_{ab}=|b\cap V(H)|\,B_{ba}$, and the same in $K$. By
connectivity the proportions agree.

(c) Match the stable cells of the $H_i$ according to their common
normalized quotient. The unions of matched cells form an equitable
partition of $H_1\sqcup\cdots\sqcup H_k$ with normalized quotient
$\Cq(H_1)$. The stable partition is coarser, and by (a) it induces the
stable partition of $H_1$, which meets each of these unions in a single
cell. So no two of the unions are merged, and they are the stable cells.
\end{proof}

\subsection{Complements and two operations}
For a graph $G$ on $n$ vertices, $A_{\overline G}=J-I-A_G$ and
$D_{\overline G}=(n-1)I-D_G$, where $J=\one\one^T$. Since $J$ maps every
subspace containing $\one$ into itself, $M_{\overline G}=M_G$. Complements
have the same equitable partitions and the same automorphisms, so
$W_{\overline G}=W_G$ and $U_{\overline G}=U_G$, and complementation
preserves AD-observability and AD-separability.

The complement formula is the case $c=n$ of the following operations on
pointed data $(A,D,s)$ with $R=ss^T$, where $c$ and $t$ are real numbers:
\[
  C_c(A,D,s)=(cR-I-A,\,(c-1)I-D,\,s),\qquad
  S_t(A,D,s)=(A+tR,\,D+tI,\,s).
\]
Thus $\Theta_{\overline G}=C_n\Theta_G$. Since $R$ maps every subspace
containing $s$ into itself, these operations do not change the cyclic
subspace generated by $s$. They satisfy
\begin{equation}\label{eq:compositions}
  C_cC_{c'}=S_{c-c'},\qquad C_cS_t=C_{c-t},\qquad S_tC_c=C_{c+t},\qquad
  S_tS_{t'}=S_{t+t'} ,
\end{equation}
so in particular $C_c$ is an involution. On normalized quotients define
$C_c(B,\pi)=(c\one\pi^T-I-B,\pi)$ and $S_t(B,\pi)=(B+t\one\pi^T,\pi)$.
Then $\Cq(\overline G)=C_n\Cq(G)$, because a vertex of $C_i$ has
$|C_j|-\delta_{ij}-B_{ij}=n\pi_j-\delta_{ij}-B_{ij}$ neighbours in $C_j$ in
the complement. The identities \eqref{eq:compositions} hold on normalized
quotients as well. These operations are purely algebraic; the
transformed data need not come from a graph.

\begin{remark}\label{rem:hom}
For a word $w=D^{a_0}AD^{a_1}A\cdots AD^{a_\ell}$, the entry of
$w(A,D)\one$ at $v$ counts homomorphisms from the caterpillar described
in the introduction that send $x_0$ to $v$, and $\mu_G(w)$ is the number of
all homomorphisms from that caterpillar to $G$. Every caterpillar arises
in this way. Hence $G\simm H$ means that $\hom(C,G)/|V(G)|=\hom(C,H)/|V(H)|$
for every caterpillar $C$. Similarly, $G\simc H$ means that
$\hom(T,G)/|V(G)|=\hom(T,H)/|V(H)|$ for every tree $T$: for graphs of equal
order this is the theorem of Dvo\v{r}\'ak~\cite{Dvorak2010} and Dell, Grohe
and Rattan~\cite{DGR2018}, and the general case follows by replacing $G$
and $H$ with $|V(H)|$ copies of $G$ and $|V(G)|$ copies of $H$.
\end{remark}

\section{Forests}\label{sec:forests}

\begin{theorem}\label{thm:forest}
Every forest is AD-observable.
\end{theorem}

\begin{proof}
Let $F$ be a forest on $n$ vertices. For an integer diagonal matrix $Z$
with entries in $\{0,\dots,n\}$ and an integer $j$, let $\sigma_j(Z)$ be the
diagonal $0/1$ matrix that selects the entries equal to $j$. It is a
polynomial in $Z$, so $\sigma_j(Z)\in\cA(F)$ whenever $Z\in\cA(F)$. We use
the following observation repeatedly: if $F'$ is a forest on $V$ with
adjacency matrix $A'$ and every vertex of $X\subseteq V$ has degree at most
one in $F'$, then
\begin{equation}\label{eq:leaf}
  A'E_XA'=\diag\bigl(|N_{F'}(v)\cap X|\bigr)_{v\in V},
\end{equation}
because a vertex of $X$ cannot be adjacent to two different vertices.

We prune leaves while keeping the vertex set fixed. At stage $t$ there is
a forest $F_t$ on $V$ with $F_0=F$, a set $V_t$ of active vertices with
$V_0=V$, and every inactive vertex is isolated in $F_t$. Every active vertex
$v$ carries the rooted tree $\beta_t(v)$, namely the component of $v$ in
the forest obtained from $F$ by deleting all edges between active
vertices, rooted at $v$; it consists of $v$ and the vertices already
removed below $v$. The \emph{type} of $v$ is the isomorphism class of this
rooted tree. Let $P_{t,\rho}$ select the active vertices of type
$\rho$. We show by induction on $t$ that the adjacency matrix $A_t$ and the
degree matrix $D_t$ of $F_t$ and all $P_{t,\rho}$ belong to $\cA(F)$. This
is clear for $t=0$, where every vertex has the type of a single vertex.

Let $R_t=\sum_\rho P_{t,\rho}$, let $\Lambda_t=R_t\sigma_1(D_t)$ select the
active leaves and let $K_t=\Lambda_t\sigma_1(A_t\Lambda_tA_t)$ select the
active leaves whose neighbour is also an active leaf, that is, the
endpoints of components of $F_t$ with two vertices; by \eqref{eq:leaf} the
matrix $A_t\Lambda_tA_t$ is diagonal. Put $\eta_t=\Lambda_t-K_t$ and
$\eta_{t,\rho}=P_{t,\rho}\eta_t$. The vertices selected by $\eta_t$ are
removed at this stage. Each of them has a unique neighbour, its parent,
which is active and is not removed. Put
\[
  g_{t,\rho}=A_t\eta_{t,\rho}A_t,\qquad
  A_{t+1}=(I-\eta_t)A_t(I-\eta_t),\qquad
  D_{t+1}=D_t-\eta_t-\sum_\rho g_{t,\rho}.
\]
By \eqref{eq:leaf}, $g_{t,\rho}$ is diagonal and counts the removed
neighbours of type $\rho$, and $A_{t+1},D_{t+1}$ are the adjacency and
degree matrices of the forest obtained by deleting the edges at removed
vertices. If an active vertex of type $\sigma$ survives and has $k_\rho$
removed neighbours of type $\rho$ for each $\rho$, its new rooted tree is
obtained by attaching the corresponding trees $\beta_t(x)$ at the root, so
its new type is determined by $\sigma$ and the numbers $k_\rho$. The vertices
with this history are selected by
$P_{t,\sigma}(I-\eta_t)\prod_\rho\sigma_{k_\rho}(g_{t,\rho})$, and
$P_{t+1,\rho'}$ is the sum of these selectors over all histories leading to
$\rho'$. This completes the induction.

As long as some component of $F_t$ restricted to the active vertices has
more than two vertices, that component has a leaf outside a two-vertex
component, so the process removes a vertex. It stops at a stage $s$ when
every active component is a single vertex or an edge. The active vertices
are then the centres of the components of $F$, since repeated removal of
all leaves of a tree leaves its centre.

The pruning is defined in terms of $F$ alone, so automorphisms of $F$
preserve every $F_t$, every $V_t$ and all types. A central vertex $v$ that
is isolated in $F_s$ has $\beta_s(v)$ equal to its component rooted at
$v$; two such vertices lie in the same orbit if and only if they have the
same type. The indicator vector of such an orbit is
$R_s\sigma_0(D_s)P_{s,\rho}\one\in M_F$. If $v$ lies on a central edge
$vv'$, the component is the union of $\beta_s(v)$ and $\beta_s(v')$ joined
by this edge. The endpoints of central edges whose own type is $\rho$ and
whose partner has type $\rho'$ are selected by
$P_{s,\rho}K_s\sigma_1(A_sP_{s,\rho'}K_sA_s)$, and they form one orbit (for
$\rho=\rho'$ the two endpoints of such an edge are exchanged by an
automorphism). So the indicator vectors of all orbits of central vertices
lie in $M_F$.

Now go back through the stages. Let $t<s$, let $O\subseteq V_{t+1}$ be an
orbit whose indicator vector lies in $M_F$, and let $\rho$ be a type. Then
\[
  \eta_{t,\rho}A_t\one_O=\one_X,
\]
where $X$ is the set of vertices removed at stage $t$ that have type $\rho$
and a parent in $O$; the entries are $0$ and $1$ because every removed vertex
has exactly one neighbour in $F_t$. If $X$ is nonempty, it is an orbit: for
$x,y\in X$ an automorphism mapping the parent of $x$ to the parent of $y$
maps $x$ to a removed neighbour of type $\rho$ of the parent of $y$, and
the subtrees hanging at the parent of $y$ from two such neighbours can be
interchanged. Every orbit of vertices removed at stage $t$ arises in this
way. By downward induction the indicator vectors of all orbits lie in
$M_F$, so $M_F=U_F$.
\end{proof}

Note that McKay's leaf pruning~\cite[Algorithm~5.4]{McKay1977} only takes
traces of diagonal matrices, and the trace of a diagonal matrix $X$ is
$\one^TX\one$; so the principal moments $\mu_T(w)$ already distinguish
non-isomorphic trees. The upward and downward passes in the proof above
follow the argument of Li, Fan, Ito, Karimi and Xu for the Terwilliger
algebra of a rooted tree~\cite{LiFanIto2021}. There the distance
projections from the root are among the generators, and the finer
selectors are built from them and $A$; here all selectors have to be built
from $A$ and $D$, and McKay's pruning provides them.

\section{Disconnected graphs}\label{sec:components}

Let $G$ be a graph with connected components $H_1,\dots,H_k$. We call the
moment equivalence classes of the $H_i$ the \emph{component types} of $G$.
For a type $\tau$ let $w_\tau$ be the proportion of vertices of $G$ lying
in components of type $\tau$.

\begin{lemma}\label{lem:components}
Let $\tau_1,\dots,\tau_r$ be the component types of $G$.
\begin{enumerate}[label=(\alph*)]
\item For each $i$ there is a matrix $Q_i\in\cA(G)$ such that $Q_i\one$ is
the indicator vector of the union of the components of type $\tau_i$.
\item $M_G=M_1\oplus\cdots\oplus M_r$, where $M_i=\cA(G)Q_i\one$ is
supported on the components of type $\tau_i$. The pointed data
$\bigl(A|_{M_i},D|_{M_i},Q_i\one/\|Q_i\one\|\bigr)$ are isometric to
$\Theta_H$ for every component $H$ of type $\tau_i$. In particular
$q(G)=\sum_iq(H_{(i)})$, where $H_{(i)}$ is any component of type
$\tau_i$.
\item $\ker(L|_{M_G})$ is spanned by $Q_1\one,\dots,Q_r\one$, where $L=D-A$.
\item If $G\simm G'$, then $G$ and $G'$ have the same component types and
$w_\tau(G)=w_\tau(G')$ for every type $\tau$.
\end{enumerate}
\end{lemma}

\begin{proof}
Let $P$ be the orthogonal projection onto $\ker L$. It is a polynomial in
$L$, so $P\in\cA(G)$, and $P=\bigoplus_j s_{H_j}s_{H_j}^T$. For every
polynomial $p$,
\begin{equation}\label{eq:compress}
  P\,p(A,D)\,P=\bigoplus_j\hmu_{H_j}(p)\,s_{H_j}s_{H_j}^T .
\end{equation}
Choose finitely many words that distinguish the types pairwise; a generic
linear combination $p$ of them takes pairwise different values
$\lambda_1,\dots,\lambda_r$ on the types. If $f_i$ is a polynomial with
$f_i(\lambda_j)=\delta_{ij}$, then $Q_i=Pf_i(Pp(A,D)P)$ is the projection
$\bigoplus_{H_j\in\tau_i}s_{H_j}s_{H_j}^T$, which proves (a).

For a word $w$, the vector $w(A,D)Q_i\one$ equals $w(A_{H_j},D_{H_j})\one$
on each component $H_j$ of type $\tau_i$ and vanishes elsewhere. By
Lemma~\ref{lem:pointed}, applied to the components of type $\tau_i$, the
linear relations between these vectors are the same as for a single
component, and the inner products are the same up to the factor
$\|Q_i\one\|^2/|V(H)|$, where $H$ is any component of type $\tau_i$. This
gives the description of $M_i$. Since
$\one=\sum_iQ_i\one$ and each $M_i$ is invariant, $M_G=\sum_iM_i$, and the
sum is direct because the supports are disjoint. Part (c) holds because
$\ker L_{H}=\R\one_H$ for connected $H$, so $\ker(L|_{M_i})=\R Q_i\one$.

For (d), choose $P$ as a polynomial in $L$ that works for $G$ and for $G'$
(using the union of their Laplacian spectra), and choose $p$ and the $f_i$
using all types occurring in $G$ or $G'$. Then $\hmu_G(Q_i)=w_{\tau_i}(G)$
and $\hmu_{G'}(Q_i)=w_{\tau_i}(G')$, and these numbers agree.
\end{proof}

\begin{proposition}\label{prop:union}
A graph $G$ is AD-observable if and only if every connected component of
$G$ is AD-observable and any two moment equivalent components of $G$ are
isomorphic. The same holds with components of $\overline G$ in place of
components of $G$.
\end{proposition}

\begin{proof}
The restriction of $M_G$ to a component $H$ is $M_H$, and the restriction
of $U_G$ is $U_H$, so observability of $G$ implies that of its
components. Automorphisms permute isomorphic components, and isomorphisms
between components extend to automorphisms of $G$; hence
$\dim U_G=\sum\dim U_H$, the sum over one component from each isomorphism
class. If all components are observable, each isomorphism class
contributes $q(H)$ to this sum, while by Lemma~\ref{lem:components} each
moment class contributes $q(H)$ to $\dim M_G$. Equality holds exactly when
every moment class contains only one isomorphism class. The last sentence
follows because $M$ and $U$ do not change under complementation.
\end{proof}

For example, $K_3$ and $C_4$ are vertex-transitive and moment equivalent,
so neither $K_3\sqcup C_4$ nor its complement $\overline{K_3}\vee 2K_2$ is
AD-observable.

\begin{lemma}\label{lem:gluing}
Suppose that every component $H$ of $G$ satisfies $M_H=W_H$, and that any
two moment equivalent components of $G$ are color equivalent. Then
$M_G=W_G$.
\end{lemma}

\begin{proof}
Fix a type $\tau_i$ and a component $H$ of this type. Every stable cell of
$H$ has an indicator vector of the form $p(A_H,D_H)\one_H$. By assumption,
all components of type $\tau_i$ have the same normalized quotient, so
$p(A,D)Q_i\one$ is
the indicator vector of the corresponding cell in all of them together.
By Lemma~\ref{lem:cr-facts}, these sets are exactly the stable cells of $G$
meeting the components of type $\tau_i$; components of different types
are not color equivalent, so they share no stable color. Hence every
stable cell of $G$ has its indicator vector in $M_G$.
\end{proof}

\section{Complements with the wrong number of vertices}\label{sec:wrong-size}

If a graph $F$ has $n$ vertices, then $C_n\Theta_F$ is the pointed data of
$\overline F$, which is connected when $F$ is disconnected. The following
theorem says that if the components of $F$ have at least two types, then
no positive value of $c$ produces the data of a graph whose components
have at least two types. It holds for all graphs.

We use the following consequences of the Perron--Frobenius theorem for a
connected graph $H$. Let $\rho_A(H)$ and $\rho_Q(H)$ be the largest
eigenvalues of $A_H$ and of the signless Laplacian $Q_H=D_H+A_H$. Both
have eigenvectors with positive entries; the corresponding spectral
projections are polynomials in $A_H$ and $Q_H$ and do not annihilate
$\one$, so $\rho_A(H)$ and $\rho_Q(H)$ are eigenvalues of $A_H$ and $Q_H$
restricted to $M_H$. Every eigenvalue of $A_H$ lies in
$[-\rho_A(H),\rho_A(H)]$. Moreover
\begin{equation}\label{eq:perron}
  \rho_Q(H)\ge\lambda_{\max}(L_H),\qquad \rho_Q(H)\ge 2\rho_A(H):
\end{equation}
for the first, $x^TL_Hx\le|x|^TQ_H|x|$ for every vector $x$, where $|x|$ is
taken entrywise; for the second, evaluate $Q_H=L_H+2A_H$ at a unit
Perron vector of $A_H$. For a component type $\tau$ we write $\rho_A(\tau)$
and $\rho_Q(\tau)$ for these numbers of any component of type $\tau$; they
are determined by the pointed data by Lemma~\ref{lem:components}.

\begin{theorem}\label{thm:wrong-size}
Let $F$ and $H$ be graphs, each with at least two component types. Then
$\Theta_H\not\cong C_c\Theta_F$ for every real number $c>0$.
\end{theorem}

\begin{proof}
Suppose that $\Theta_H\cong C_c\Theta_F$, and identify $M_H$ with
$M=M_F$ by the isometry, so that $s_H=s_F=s$. Then on $M$
\[
  A_H=cR-I-A_F,\quad D_H=(c-1)I-D_F,\quad
  L_H=c(I-R)-L_F,\quad Q_H=cR+(c-2)I-Q_F .
\]
Since $C_c$ is an involution, also $\Theta_F\cong C_c\Theta_H$, and every
statement below remains true with $F$ and $H$ exchanged.

For $x=\alpha s+y$ with $y\perp s$ we have $L_Hx=0$ if and only if
$L_Fy=cy$, because $L_Fs=0$. Hence
\begin{equation}\label{eq:kernel}
  \ker(L_H|_M)=\R s\oplus E_c ,
\end{equation}
where $E_c$ is the $c$-eigenspace of $L_F$ on $M$. By
Lemma~\ref{lem:components}(c) the left side has dimension at least two, so
$c$ is an eigenvalue of $L_F$ on $M$.

Since $Q_H$ is positive semidefinite, $Q_F\preceq cR+(c-2)I$ on $M$. The
matrix on the right has the eigenvalue $c-2$ on $s^\perp$, so the second
largest eigenvalue of $Q_F$ on $M$ is at most $c-2$. By
Lemma~\ref{lem:components}, each component type $\tau$ of $F$ contributes
the eigenvalue $\rho_Q(\tau)$ of $Q_F$ on $M$, with an eigenvector
supported on the components of type $\tau$; these eigenvectors are
orthogonal. Hence at most one type has $\rho_Q(\tau)>c-2$. A type for which
$c$ is a Laplacian eigenvalue has $\rho_Q(\tau)\ge c$ by \eqref{eq:perron},
and some type has this property. Therefore exactly one type $\tau_*$ of $F$
carries the Laplacian eigenvalue $c$, so that $E_c$ is supported on the
components of type $\tau_*$, and every other type $\tau$ of $F$ satisfies
\begin{equation}\label{eq:passive}
  \rho_A(\tau)\le\tfrac12\rho_Q(\tau)\le\tfrac{c-2}2 .
\end{equation}
The same holds for $H$, and we write $\sigma_*$ for the type of $H$ that
carries the Laplacian eigenvalue $c$ of $L_H$. In particular $c\ge2$.

Choose a type $\tau\ne\tau_*$ of $F$ and a type $\sigma\ne\sigma_*$ of $H$.
Exchanging $F$ and $H$ if necessary, we may assume that
$a=\rho_A(\tau)\ge b=\rho_A(\sigma)$. Let $T$ be the orthogonal projection
of $M$ onto the summand of type $\tau$ in Lemma~\ref{lem:components}, and
put $t=Ts$. Let $V_\sigma\subseteq M$ be the summand of $H$ of type
$\sigma$, let $P$ be the projection onto it, and put $h=Ps$ and
$v=\|h\|^2>0$. Viewed in $\R^{V(H)}$, the vector $h$ is a multiple of the
indicator vector of the components of type $\sigma$, so $h\in\ker L_H$, and
$\langle s,h\rangle=v$. By \eqref{eq:kernel}, $h-vs\in E_c$, which is
orthogonal to the summand of type $\tau$. Hence
\begin{equation}\label{eq:Th}
  Th=vt .
\end{equation}

Let $A_\tau$ be the restriction of $A_F$ to the summand of type $\tau$,
let $A_\sigma$ be the restriction of $A_H$ to $V_\sigma$, and let
$\Psi\colon V_\sigma\to TM$ be the restriction of $T$. For $x\in V_\sigma$
we have $\langle s,x\rangle=\langle h,x\rangle$, and $T$ commutes with
$A_F$. Applying $T$ to $A_Hx=c\langle s,x\rangle s-x-A_Fx$ gives
\begin{equation}\label{eq:sylvester}
  \Psi A_\sigma+(I+A_\tau)\Psi=c\,t\,h^T .
\end{equation}
Let $u$ be the normalized orthogonal projection of $t$ onto the
eigenspace of $A_\tau$ for the eigenvalue $a$; by the remarks before the
theorem this projection is not zero, and $\gamma=\langle u,t\rangle>0$.
Multiplying \eqref{eq:sylvester} by $u^T$ from the left gives
$u^T\Psi B=c\gamma h^T$ with $B=(a+1)I+A_\sigma$. The eigenvalues of
$A_\sigma$ lie in $[-b,b]$, so
\[
  I\preceq(a+1-b)I\preceq B\preceq(a+b+1)I\preceq(c-1)I ,
\]
the last step by \eqref{eq:passive}. Hence $B$ is invertible and, using
\eqref{eq:Th},
\[
  v\gamma=u^T\Psi h=c\gamma\,h^TB^{-1}h\ge\frac{c\gamma v}{c-1}.
\]
Dividing by $\gamma v>0$ gives $c\le c-1$, a contradiction.
\end{proof}

\section{Cographs}\label{sec:cographs}

A graph is a cograph if it can be built from single vertices by disjoint
unions and complementation; equivalently, it has no induced path on four
vertices~\cite{CorneilLerchsStewart1981}. Induced subgraphs and complements
of cographs are cographs, and a connected cograph with at least two
vertices has a disconnected complement. A cograph $G$ is represented by
its \emph{reduced cotree}: the leaves are the vertices, every internal
node is labelled union or join and has at least two children, labels
alternate along every path from the root, and $G$ is obtained by
evaluating the tree. The subgraph $G_x$ induced by the leaves below a
node $x$ is the union or the join of the graphs $G_y$ for the children $y$
of $x$; these are connected below a union node, and have connected
complements below a join node.

\subsection{Following the cotree downwards}

\begin{lemma}\label{lem:descent}
Let $G$ be a connected cograph with $q(G)\ge2$, and let $\alpha$ and $\beta$
be the smallest positive and the largest eigenvalue of $L_G$ on $M_G$. Put
$G_0=G$. If the components of $\overline{G_j}$ have at least two types,
stop; otherwise let $G_{j+1}$ be any component of $\overline{G_j}$. This
process stops, say at $G_\ell$, and $F=\overline{G_\ell}$ is a cograph
with at least two component types, each of which has $q<q(G)$. Moreover
\[
  \Theta_F=S_{-\alpha}\Theta_G\qquad\text{or}\qquad\Theta_F=C_\beta\Theta_G .
\]
\end{lemma}

\begin{proof}
If the components of $\overline{G_j}$ all have one type, then
$\overline{G_j}$ has the same normalized moments as each of its
components, so $\Theta_{G_{j+1}}=C_{n_j}\Theta_{G_j}$ with $n_j=|V(G_j)|$.
In particular $q(G_{j+1})=q(G)\ge2$, so $G_{j+1}$ is a connected cograph
with at least two vertices and its complement is disconnected. The orders
decrease, so the process stops. By \eqref{eq:compositions},
$\Theta_F=C_{n_\ell}\cdots C_{n_0}\Theta_G$ equals $S_t\Theta_G$ or
$C_u\Theta_G$ for some real $t$ or $u$. On the orthogonal complement of $s$
in $M_G$, the operation $S_t$ turns $L_G$ into $L_G+tI$ and $C_u$ turns it
into $uI-L_G$. The eigenvalues of $L_G$ there lie in $[\alpha,\beta]$ with
$\alpha>0$, since $G$ is connected. On the other hand $L_F$ is positive
semidefinite and, by Lemma~\ref{lem:components}(c), has a zero eigenvalue
on the orthogonal complement of $s$. This forces $t=-\alpha$ or $u=\beta$.
The statement on dimensions is Lemma~\ref{lem:components}(b).
\end{proof}

\subsection{Proof of the main results}

\begin{proposition}\label{prop:key}
If two connected cographs are moment equivalent, then they are color
equivalent.
\end{proposition}

\begin{proof}
We use induction on $q$, the common dimension of the walk spaces, and for
fixed $q$ induction on the sum of the orders of the two graphs. If $q=1$,
both graphs are regular of the same degree $d$ and both normalized
quotients are $((d),(1))$.

Let $q\ge2$, and let $G$ and $K$ be connected cographs with $G\simm K$,
$q(G)=q(K)=q$ and $|V(G)|+|V(K)|=N$. Assume the proposition for all pairs
with smaller $q$, and for pairs with the same $q$ and smaller sum of
orders. The two graphs have the same numbers $\alpha$ and $\beta$, because
these are eigenvalues of $L$ on the walk space. Apply
Lemma~\ref{lem:descent} to both graphs, with outcomes $F_G$ and $F_K$.

At every step $j$ of the process for $G$ before the last one, the
components of $\overline{G_j}$ are connected cographs of one type with
$q=q(G)$, and two of them have at most $|V(G)|<N$ vertices together. By
the inner induction they are pairwise color equivalent, so
$\Cq(G_{j+1})=\Cq(\overline{G_j})=C_{n_j}\Cq(G_j)$ by
Lemma~\ref{lem:cr-facts}(c). Hence $\Cq(F_G)$ is obtained from $\Cq(G)$ by
the same composition of operations as $\Theta_{F_G}$ from $\Theta_G$, and
the same holds for $K$.

Suppose first that both processes end with the same operation, say
$\Theta_{F_G}=E\Theta_G$ and $\Theta_{F_K}=E\Theta_K$ with
$E\in\{S_{-\alpha},C_\beta\}$. Then $F_G\simm F_K$, so by
Lemma~\ref{lem:components}(d) the two graphs have the same component types
with the same proportions. Each type has $q$ smaller than $q(G)$, so by
the outer induction all components of a given type, in $F_G$ and in
$F_K$, are color equivalent. Components of different types share no
stable color. Hence $\Cq(F_G)=\Cq(F_K)$, and applying the inverse of $E$
gives $\Cq(G)=\Cq(K)$.

Otherwise we may assume $\Theta_{F_G}=S_{-\alpha}\Theta_G$ and
$\Theta_{F_K}=C_\beta\Theta_K$. Then
$\Theta_{F_K}\cong C_\beta S_\alpha\Theta_{F_G}=C_{\beta-\alpha}\Theta_{F_G}$.
If $\beta=\alpha$, this gives $D_{F_K}=-I-D_{F_G}$ on the walk space, which is
impossible since both degree matrices are positive semidefinite. If
$\beta>\alpha$, it contradicts Theorem~\ref{thm:wrong-size}.
\end{proof}

\begin{theorem}\label{thm:main}
Every cograph $G$ satisfies $M_G=W_G$. Moreover, if two cographs are moment
equivalent, then they are color equivalent.
\end{theorem}

\begin{proof}
We prove $M_G=W_G$ by induction on $|V(G)|$; it is clear for one vertex. If
$G$ is disconnected, its components are smaller cographs, so they satisfy
$M=W$, and moment equivalent components are color equivalent by
Proposition~\ref{prop:key}; Lemma~\ref{lem:gluing} gives $M_G=W_G$. If $G$
is connected with at least two vertices, the same argument applies to the
disconnected cograph $\overline G$, and $M_G=M_{\overline G}$,
$W_G=W_{\overline G}$.

For the second statement, let $G\simm G'$. By
Lemma~\ref{lem:components}(d) the two graphs have the same component types
with the same proportions, and by Proposition~\ref{prop:key} all
components of one type are color equivalent. Components of different
types share no stable color, so $\Cq(G)=\Cq(G')$.
\end{proof}

\begin{proof}[Proof of Theorem~\ref{thm:intro-hom}]
Since $G$ and $H$ have the same order, equal caterpillar counts mean
$G\simm H$ by Remark~\ref{rem:hom}. By Theorem~\ref{thm:main}, $G\simc H$,
and for graphs of the same order this means that $G$ and $H$ have the same
tree counts~\cite{Dvorak2010,DGR2018}.
\end{proof}

\begin{remark}
The example from the introduction shows why the induction in
Proposition~\ref{prop:key} has to allow different orders. Let
$F=K_2\sqcup K_1$. The graphs $F\vee F\vee F$ and $(F\sqcup F)\vee(F\sqcup F)$
have $9$ and $12$ vertices. In both, the stable cells are the vertices of
degree $6$ and $7$, with proportions $1/3$ and $2/3$, and both normalized
quotients equal
\[
  \left(\begin{pmatrix}2&4\\2&5\end{pmatrix},\ \Bigl(\tfrac13,\tfrac23\Bigr)\right).
\]
So the graphs are color equivalent and in particular moment equivalent,
but their complements $3P_3$ and $2(P_3\vee P_3)$ are not even moment
equivalent: their Laplacians on the walk space have spectra $\{0,3\}$ and
$\{0,6\}$.
\end{remark}

\section{Observable cographs}\label{sec:observable}

\begin{proposition}\label{prop:refinable}
A cograph is refinable if and only if every internal node of its reduced
cotree has the property stated in Theorem~\ref{thm:intro-observable}.
\end{proposition}

\begin{proof}
Let $x$ be an internal node with children $y_1,\dots,y_k$, and put
$X=G_x$ and $X_i=G_{y_i}$. If $x$ is a union node, the $X_i$ are the
components of $X$. If it is a join node, the complements $\overline{X_i}$
are the components of $\overline X$. Complementation preserves stable
partitions, automorphisms and isomorphism, so the property at $x$ and the
refinability of $X$ and of the $X_i$ do not change if we replace these
graphs by their complements. It therefore suffices to consider unions.
By Lemma~\ref{lem:cr-facts}(a), the stable partition of $X$ restricts to
the stable partition of each $X_i$. Automorphisms of $X$ permute
components, every isomorphism between components extends to an
automorphism of $X$, and two vertices of one component lie in the same
orbit of $X$ if and only if they lie in the same orbit of the component.

If $X$ is refinable, each stable cell of $X_i$ is the intersection of an
orbit of $X$ with $V(X_i)$, hence an orbit of $X_i$, so $X_i$ is refinable.
A stable cell of $X$ meeting two components is an orbit, so these
components are isomorphic. Conversely, suppose that every $X_i$ is
refinable and that children sharing a stable color are isomorphic. A
stable cell $Z$ of $X$ meets only components isomorphic to each other. An
isomorphism between two of them extends to an automorphism of $X$, which
preserves stable colors and therefore maps $Z\cap V(X_i)$ to $Z\cap V(X_j)$.
Within one component, $Z$ is a stable cell and hence an orbit. So $Z$ is an
orbit of $X$. Induction from the leaves proves the proposition.
\end{proof}

\begin{proof}[Proof of Theorem~\ref{thm:intro-observable}]
By Theorem~\ref{thm:main} and Section~\ref{sec:prelim}, a cograph is
AD-observable if and only if it is AD-separable, and if and only if it is
refinable. Apply Proposition~\ref{prop:refinable}.
\end{proof}

The cotree has to be reduced. The graph $\overline{K_3}$ is the union of
$\overline{K_2}$ and $K_1$, and these two children share their only stable
color without being isomorphic, although $\overline{K_3}$ is
vertex-transitive.

\begin{example}
A threshold graph is built from one vertex by adding isolated or dominating
vertices. Two of its vertices with the same degree are twins, so its degree
classes are orbits. The vectors $E_X\one$ for the degree classes $X$ are
polynomials in $D$ applied to $\one$, so threshold graphs are AD-observable.
The complement of a complete multipartite graph is a disjoint union of
complete graphs, and two complete graphs are moment equivalent only if
they have the same order. By Proposition~\ref{prop:union}, complete
multipartite graphs are AD-observable. On the other hand, the root of the
cotree of $\overline{K_3\sqcup C_4}=\overline{K_3}\vee2K_2$ has the
children $\overline{K_3}$ and $2K_2$, which share the only stable color of
this $4$-regular graph.
\end{example}

\section{Beyond cographs}\label{sec:beyond}

\subsection{Distance-hereditary graphs}\label{sec:dh}
Let $G_1$ be the graph obtained from a path $p_1p_2p_3p_4$ by adding two
adjacent vertices $x,x'$ joined to $p_1$, two adjacent vertices $y,y'$
joined to $p_4$, and a vertex $z$ joined to $p_4$. Let $G_2$ be $G_1$ with
the additional edge $p_1p_4$; see Figure~\ref{fig:dh}. Both graphs are
distance-hereditary: deleting $z$, then $x'$ and $y'$ (twins of $x$ and
$y$), then the pendant vertices $x$ and $y$ leaves a path or a $4$-cycle.

\begin{figure}[ht]
\centering
\begin{tikzpicture}[scale=0.95,every node/.style={circle,fill,inner sep=1.6pt},
  lab/.style={draw=none,fill=none,inner sep=1pt}]
\begin{scope}
  \node (p1) at (0,0) {}; \node (p2) at (1.2,0) {}; \node (p3) at (2.4,0) {};
  \node (p4) at (3.6,0) {};
  \node (x) at (-0.9,0.7) {}; \node (xx) at (-0.9,-0.7) {};
  \node (y) at (4.5,0.7) {}; \node (yy) at (4.5,-0.7) {};
  \node (z) at (3.6,-1.0) {};
  \draw (x)--(xx)--(p1)--(x); \draw (p1)--(p2)--(p3)--(p4);
  \draw (y)--(yy)--(p4)--(y); \draw (p4)--(z);
  \node[lab] at (0,0.35) {$p_1$}; \node[lab] at (1.2,0.35) {$p_2$};
  \node[lab] at (2.4,0.35) {$p_3$}; \node[lab] at (3.6,0.35) {$p_4$};
  \node[lab] at (-1.2,0.95) {$x$}; \node[lab] at (-1.2,-0.95) {$x'$};
  \node[lab] at (4.8,0.95) {$y$}; \node[lab] at (4.8,-0.95) {$y'$};
  \node[lab] at (3.95,-1.1) {$z$};
  \node[lab] at (1.8,-1.6) {$G_1$};
\end{scope}
\begin{scope}[xshift=7.6cm]
  \node (p1) at (0,0) {}; \node (p2) at (1.2,0) {}; \node (p3) at (2.4,0) {};
  \node (p4) at (3.6,0) {};
  \node (x) at (-0.9,0.7) {}; \node (xx) at (-0.9,-0.7) {};
  \node (y) at (4.5,0.7) {}; \node (yy) at (4.5,-0.7) {};
  \node (z) at (3.6,-1.0) {};
  \draw (x)--(xx)--(p1)--(x); \draw (p1)--(p2)--(p3)--(p4);
  \draw (y)--(yy)--(p4)--(y); \draw (p4)--(z);
  \draw (p1) to[out=-55,in=-125] (p4);
  \node[lab] at (0,0.35) {$p_1$}; \node[lab] at (1.2,0.35) {$p_2$};
  \node[lab] at (2.4,0.35) {$p_3$}; \node[lab] at (3.6,0.35) {$p_4$};
  \node[lab] at (-1.2,0.95) {$x$}; \node[lab] at (-1.2,-0.95) {$x'$};
  \node[lab] at (4.8,0.95) {$y$}; \node[lab] at (4.8,-0.95) {$y'$};
  \node[lab] at (3.95,-1.1) {$z$};
  \node[lab] at (1.8,-1.6) {$G_2$};
\end{scope}
\end{tikzpicture}
\caption{The two distance-hereditary graphs on nine vertices with $M_G\ne W_G$.}
\label{fig:dh}
\end{figure}

\begin{proposition}\label{prop:dh}
For $G\in\{G_1,G_2\}$ the stable cells are $\{x,x'\}$, $\{y,y'\}$ and the
five singletons; they are also the orbits. The vector $f$ with
$f(x)=f(x')=f(y)=f(y')=1$, $f(p_2)=f(p_3)=-2$ and $f=0$ elsewhere lies in
$W_G$ but not in $M_G$. The graph $G$ is AD-separable but not
AD-observable.
\end{proposition}

\begin{proof}
In $G_1$ the degree-$2$ vertices are $x,x',p_2,p_3,y,y'$. The neighbours of
$x,x',p_2$ have degrees $\{2,3\}$ and those of $p_3,y,y'$ have degrees
$\{2,4\}$; one more round separates $p_2$ from $x,x'$ and $p_3$ from
$y,y'$. The remaining vertices $p_1,p_4,z$ have distinct degrees. The
interchanges $x\leftrightarrow x'$ and $y\leftrightarrow y'$ are
automorphisms. The same computation works for $G_2$, where $p_1$ and $p_4$
have degrees $4$ and $5$.

A direct check gives $Af=f$ and $Df=2f$, and $\langle f,\one\rangle=0$. So
$\R f$ is invariant under $A$ and $D$ and orthogonal to $\one$, and since
$A$ and $D$ are symmetric, $f$ is orthogonal to $w(A,D)\one$ for every word
$w$. Thus $f\perp M_G$, while $f$ is constant on the stable cells.

For separability it suffices to look at walk counts. Let $d=A\one$. The
vertices $z,p_1,p_4$ have distinct degrees, different from $2$. On the
degree-$2$ vertices $x,p_2,p_3,y$ the pairs $\bigl((Ad)(v),(A^2d)(v)\bigr)$
are
\[
  (5,11),\ (5,12),\ (6,12),\ (6,13)\ \text{in }G_1,\qquad
  (6,17),\ (6,18),\ (7,17),\ (7,18)\ \text{in }G_2,
\]
and $x',y'$ have the same values as $x,y$. So the vectors $d$, $Ad$ and
$A^2d$ separate the seven orbits.
\end{proof}

In fact $\dim M_G=6$ for both graphs, since the vectors
$\one,A\one,\dots,A^5\one$ are linearly independent, while $W_G=U_G$ has
dimension $7$. We generated all graphs with at most nine vertices up to
isomorphism with nauty~\cite{McKayPiperno2014} and computed $M_G$ in
exact rational arithmetic for those that are distance-hereditary,
connected or not. Up to isomorphism, $G_1$ and $G_2$ are the only ones
with $M_G\ne W_G$. The code for the computations in this section is
included as an ancillary file of the arXiv version. Removing $z$ from
$G_1$ gives an eight-vertex graph on which
color refinement merges the six vertices of degree $2$; there $M=W$, but
the graph is not refinable.

\subsection{General graphs}
Let $G_{10}$ be the graph on $\{0,1,\dots,9\}$ with edges
\[
\begin{gathered}
04,\ 06,\ 08,\ 15,\ 17,\ 18,\ 24,\ 25,\ 29,\ 36,\\
37,\ 39,\ 47,\ 48,\ 56,\ 59,\ 67,\ 68,\ 79,\ 89,
\end{gathered}
\]
and let $H_{10}$ be obtained by replacing the edges $04,15$ with $05,14$.

\begin{proposition}\label{prop:ten}
The graphs $G_{10}$ and $H_{10}$ have the same principal moments but are
not color equivalent. Moreover $\dim M_{G_{10}}=5$, and $G_{10}$ is
refinable but not AD-separable.
\end{proposition}

\begin{proof}
Let $d=A\one$ be the degree vector of $G_{10}$. One checks that the
vectors $\one$, $d$, $Ad$, $Dd$ and $A^2d$, namely
\[
\begin{aligned}
d&=(3,3,3,3,4,4,5,5,5,5), & Ad&=(14,14,13,15,16,16,20,20,20,20),\\
Dd&=(9,9,9,9,16,16,25,25,25,25), &
A^2d&=(56,56,52,60,67,67,85,85,84,84),
\end{aligned}
\]
span a five-dimensional space that is invariant under $A$ and $D$; it is
therefore $M_{G_{10}}$. Every vector in it has equal entries at $0$ and $1$
and at $4$ and $5$. The switching does not change degrees, and the
difference $A_{H_{10}}-A_{G_{10}}$ maps a vector $x$ to a vector supported
on $\{0,1,4,5\}$ whose entries are $\pm(x_5-x_4)$ and $\pm(x_1-x_0)$. So it
vanishes on $M_{G_{10}}$, and $M_{G_{10}}$ is invariant under $A_{H_{10}}$
and $D_{H_{10}}$, which act on it as $A_{G_{10}}$ and $D_{G_{10}}$ do. Hence
$M_{H_{10}}=M_{G_{10}}$ with the same action, and all principal moments
agree.

A direct computation shows that color refinement on $G_{10}$ produces
partitions with $3,5,7,9$ and $10$ cells. Hence the stable cells and the
orbits of $G_{10}$ are singletons, so $G_{10}$ is refinable, while the
vertices $0$ and $1$ are AD-equivalent. Since the stable partition of
$G_{10}$ is discrete, its quotient matrix is the adjacency matrix itself,
so every graph color equivalent to $G_{10}$ is isomorphic to it. But
$G_{10}$ and $H_{10}$ have $5$ and $7$ triangles, respectively, so they
are not isomorphic and hence not color equivalent.
\end{proof}

So caterpillar counts do not determine tree counts on all graphs. Together
with Proposition~\ref{prop:dh}, this example also shows that the classes
of AD-observable, AD-separable and refinable graphs are all different.

\subsection{A local condition}
The vector $f$ in Proposition~\ref{prop:dh} lives in one degree class,
which contains four stable cells. The next proposition shows that
$M_G\ne W_G$ is only possible if some degree class contains at least three
stable cells.

\begin{proposition}\label{prop:two-cells}
If every degree class of $G$ contains at most two stable cells, then
$M_G=W_G$. In particular, such a graph is AD-observable if and only if it
is AD-separable, and if and only if it is refinable.
\end{proposition}

\begin{proof}
Let $E_d$ select the vertices of degree $d$. Each $E_d$ is a polynomial in
$D$, so $M_G=\bigoplus_dE_dM_G$ and $E_d\one\in M_G$. If the degree class
of $d$ is one stable cell, then $E_dM_G=\R E_d\one$. If it consists of two
stable cells, $E_dM_G$ is either $\R E_d\one$ or the span of the two cell
indicators. So $M_G$ is the span of the indicator vectors of a partition
obtained from the stable partition by merging some pairs of cells. This
partition is equitable because $M_G$ is $A$-invariant, so it refines the
stable partition. Being also coarser, it equals the stable partition, and
no merging takes place.
\end{proof}

\subsection{Questions}
We end with three questions. A graph is $P_4$-sparse if every five of its
vertices induce at most one path on four vertices. These graphs include
the cographs, form a class closed under complementation and induced
subgraphs, and have a tree representation that generalizes the
cotree~\cite{JamisonOlariu1992}. The computer search of
Section~\ref{sec:dh} also covers them: every $P_4$-sparse graph with at
most nine vertices satisfies $M_G=W_G$.

\begin{question}
Does every $P_4$-sparse graph $G$ satisfy $M_G=W_G$?
\end{question}

\begin{question}
Let $G$ be a cograph and $H$ an arbitrary graph with the same number of
vertices and the same caterpillar counts. Are $G$ and $H$ color
equivalent?
\end{question}

\begin{question}
Which distance-hereditary graphs satisfy $M_G=W_G$? In both examples of
Proposition~\ref{prop:dh}, $W_G$ contains an eigenvector of $A$ supported
in one degree class and orthogonal to $\one$. Does every distance-hereditary
graph with $M_G\ne W_G$ have such a vector?
\end{question}

\section*{AI disclosure}
The results on cographs (Theorems~\ref{thm:intro-main}--\ref{thm:intro-hom}
and Sections~\ref{sec:components}--\ref{sec:observable}) were obtained in
collaboration with the AI assistants Claude Code (Anthropic) and Codex
(OpenAI), which were also used to write the code for the computations in
Section~\ref{sec:beyond}. The remaining material is based on the authors'
notes. The paper was edited and proofread with the assistance of AI tools.
The authors take full responsibility for its content.

\bibliographystyle{plain}
\bibliography{references}

\end{document}